\documentclass[12pt,final,1p,times,review]{elsarticle}

\usepackage{amssymb}
\usepackage{amsthm}
\usepackage{hyperref}
\usepackage{amsmath}
\usepackage{amsmath,amssymb,amsopn,amsfonts,mathrsfs,amsbsy,amscd}
\usepackage{longtable}
\usepackage{newtxtext, newtxmath}
\usepackage{geometry}
\journal{???}

\newcommand{\e}{\epsilon}

\newtheorem{definition}{Definition}[section]
\newtheorem{theorem}{Theorem}[section]

\newtheorem{lemma}{Lemma}[section]

\numberwithin{equation}{section}

\begin{document}
\bibliographystyle{amsplain}
\begin{frontmatter}
\title{\textbf{Rota-Baxter Operators on Dual Quaternion Algebra}}

\author[label1]{Kamran Shakoor\corref{mycorrespondingauthor}}
\cortext[mycorrespondingauthor]{Kamran Shakoor}
\address[label1]{Abdus Salam School of Mathematical Sciences,\\ GC University, Lahore-Pakistan\\e-mail:
			kamranshakoor@sms.edu.pk}
	
\author[label2]{Azhar Farooq}
\address[label2]{Abdus Salam School of Mathematical Sciences,\\ GC University, Lahore-Pakistan\\e-mail:
			azhar.farooq@sms.edu.pk}
		
\author[label3]{Hassan Oubba}
\address[label3]{Universit\'e Hassan II,
			Facult\'e des sciences Ain Chock \\  D\'epartement de Math\'ematiques et Informatique\\
			B.P. 5366 Maarif, Casablanca, Maroc\\e-mail:
			hassan.oubba@edu.umi.ac.ma}

\begin{abstract}
The purpose of this paper is to determine all Rota-Baxter operators on dual quaternion algebra $\mathcal{H}_d$ over the reals.
\end{abstract}
\end{frontmatter}

\section{Introduction} \label{Introduction}

Rota-Baxter algebras were first formulated by Glen E. Baxter \cite{Bax} in 1960, motivated by problems within probability theory. Later, it gained prominence through the foundational works of Rota \cite{Rota1, Rota2} and his collaborators. In recent decades, the Rota-Baxter algebraic framework has proven instrumental in resolving analytical and combinatorial challenges, while also finding broad applications across diverse fields of mathematics and mathematical physics (see \cite{Chu, Guo, Gao, Tang1, Tang2} and references therein).

Classifying Rota–Baxter operators on particular algebras stands as a central concern of the theory. Significant progress has been made for low-dimensional cases: Rota–Baxter operators on $2$ and $3$-dimensional algebras have been thoroughly investigated in \cite{An}, \cite{Guo2}, and \cite{Li}. For the second-order full-matrix algebra over complex fields with weight $0$, explicit classifications were derived using both standard computational methods and Gröbner basis techniques \cite{Tang2}. Furthermore, the four-dimensional Sweedler algebra, taken as an associative algebra, has its Rota–Baxter operators exhaustively characterized in \cite{Ma2,Ma1}.

Rota–Baxter operators of arbitrary weight on the split semiquaternion algebra have recently been fully classified by Chen and Deng \cite{chen}. Here, we pursue the corresponding classification for the dual quaternion algebra.

\section{Preliminaries}

\subsection{Dual Quaternion Algebra}

Quaternions, devised by Sir William Rowan Hamilton (1805–1865) in 1843 as an extension of the complex numbers, satisfy Hamilton's celebrated relation:
\[
i^2 = j^2 = k^2 = -1, \quad ijk = -1.
\]

Split quaternions, on the other hand, adhere to the conditions:
\[
i^2 = -1, \quad j^2 = k^2 = 1, \quad ijk = 1.
\]

For generalized quaternions, the parameters $\alpha$ and $\beta$ come into play, defining them as:
\[
i^2 = -\alpha, \quad j^2 = -\beta, \quad k^2 = -\alpha\beta, \quad ijk = -\alpha\beta.
\]

In 1873, William Kingdon Clifford (1845-1879) introduced dual quaternions \cite{CK} to seamlessly integrate rotations and translations while preserving the advantages of quaternion representations of rotations.

A dual quaternion $x$ takes the form of a linear combination:
\[
x = x_0e_0 + x_1e_1 + x_2e_2 + x_3e_3,
\]
where $x_0, x_1, x_2, x_3$ are real numbers and $e_0, e_1, e_2, e_3$ are basis elements. The quaternion sum follows standard componentwise addition and multiplication, is defined such that $e_0$ serves as the identity, while $e_1, e_2,e_3$ satisfy:
\begin{equation}\label{eq2.1}
e_i^2=0~~~\text{and}~~~e_i.e_j=0~~\text{for}~~1\leq i<j\leq 3.
\end{equation}
Denote the set of dual quaternions as $\mathcal{H}_d$:
$$
\mathcal{H}_d=\displaystyle{\lbrace x=\sum_{i=0}^3x_ie_i \vert
\hspace{0.2 cm}
x_0,x_1,x_2,x_3 \in \mathbb{R}, \hspace{0.2 cm} e_i^2=0,~i=1,2,3,~~e_i.e_j=0,~~1\leq i<j\leq 3 \rbrace}.
$$
We can also write $x=x_0e_0+\tilde{x}$, where $x_0e_0$ is the \textit{real part} of $x$ and $\tilde{x}=x_1e_1+x_2e_2+x_3e_3$, called the \textit{pure dual quaternion part} of $x$.\\
The quaternion product on $\mathcal{H}_d$ is defined by
\begin{equation}
x.y=(x_0y_0)e_0+(x_0y_1+y_0x_1)e_1+(x_0y_2+y_0x_2)e_2+(x_0y_3+y_0x_3)e_3
\end{equation}
for all $x,y \in \mathcal{H}_d$.\\

The set $\mathcal{H}_d$ constitutes a commutative ring under dual quaternion multiplication. It also serves as a vector space of dimension four over $\mathbb{R}$ with basis $\lbrace e_0, e_1, e_2, e_3 \rbrace$. One intriguing property of dual quaternions is their ability to express Galilean transformations through a single quaternion equation. For a more in-depth understanding of dual quaternions, we recommend referring to \cite{BA, oubd, MV}, and \cite{PP}.

We adopt the following conventions: $\mathbb{R}$ is the field of real numbers; all algebras and linear maps are over $\mathbb{R}$; and given a matrix $M$, $M^T$ denotes the transpose of $M$, and
$$
M(4) =
\begin{pmatrix}
M & 0 & 0 & 0 \\
0 & M & 0 & 0 \\
0 & 0 & M & 0 \\
0 & 0 & 0 & M
\end{pmatrix}
.
$$

\begin{definition}
For a fixed $\lambda\in \mathbb{R}$, a \textit{Rota-Baxter algebra of weight} $\lambda$ is defined as a pair $(A, \mathcal{R})$, where $A$ is an algebra and $\mathcal{R} \colon A \to A$ is a linear operator such that the \textit{Rota-Baxter equation}
\begin{equation}\tag{2.3}
\mathcal{R}(x)\mathcal{R}(y) = \mathcal{R}\big(\mathcal{R}(x)y\big) + \mathcal{R}\big(x\mathcal{R}(y)\big) + \lambda \mathcal{R}(xy), \quad \forall x, y \in A,
\end{equation}
holds. Then $\mathcal{R}$ is called a \textit{Rota-Baxter operator of weight} $\lambda$.
\end{definition}

\section{Rota-Baxter Operator on Dual Quaternion Algebra}
Let $\mathcal{R}$ be a linear transformation on $H_{d}$, and
\[
R = \begin{pmatrix}
a_{11} & a_{12} & a_{13} & a_{14} \\
a_{21} & a_{22} & a_{23} & a_{24} \\
a_{31} & a_{32} & a_{33} & a_{34} \\
a_{41} & a_{42} & a_{43} & a_{44}
\end{pmatrix}
= (\gamma_0, \gamma_1, \gamma_2, \gamma_3)
\]
the matrix of $\mathcal{R}$ with respect to the basis $\{e_0, e_1, e_2, e_3\}$. Notice easily that $\gamma_i$ is the coordinate of $\mathcal{R}$ with respect to the basis $\{e_0, e_1, e_2, e_3\}$. Thus we have
\[
\mathcal{R}(e_i) = (e_0, e_1, e_2, e_3)\gamma_i = \gamma_i^T
\begin{pmatrix}
e_0 \\
e_1 \\
e_2 \\
e_3
\end{pmatrix}.
\]
\begin{lemma}\label{lem1}
Let $\mathcal{R}$ be a linear transformation on $\mathcal{H}_{d}$. The following equation holds:
\begin{equation}\label{eq}
    \mathcal{R}(e_i)\mathcal{R}(e_j) = \gamma_i^T C \gamma_j(4)
    \begin{pmatrix}
    e_0 \\
    e_1 \\
    e_2 \\
    e_3
    \end{pmatrix},
\end{equation}
where
\[
C=\begin{pmatrix}
    1&0&0&0\\
    0&0&0&0\\
     0&0&0&0\\
      0&0&0&0\\
       0&1&0&0\\
        1&0&0&0\\
         0&0&0&0\\
          0&0&0&0\\
           0&0&1&0\\
            0&0&0&0\\
             1&0&0&0\\
              0&0&0&0\\
               0&0&0&1\\
                0&0&0&0\\
                 0&0&0&0\\
                  1&0&0&0\\

\end{pmatrix}^T
.
\]
\end{lemma}
\begin{proof}
For any $i$ and $j$, since
\[
\mathcal{R}(e_i)\mathcal{R}(e_j) = \gamma_i^T
\begin{pmatrix}
e_0 \\
e_1 \\
e_2 \\
e_3
\end{pmatrix}
(e_0, e_1, e_2, e_3)\gamma_j
= \gamma_i^T
\begin{pmatrix}
e_0 & e_1 & e_2 & e_3 \\
e_1 & 0 & 0 & 0 \\
e_2 & 0 & 0 & 0 \\
e_3 & 0 & 0 & 0
\end{pmatrix}
\gamma_j
\]
\[
= \gamma_i^T
\begin{pmatrix}
1 & 0 & 0 & 0 \\
0 & 0 & 0 & 0 \\
0 & 0 & 0 & 0 \\
0 & 0 & 0 & 0
\end{pmatrix}
\gamma_j e_0 + \gamma_i^T
\begin{pmatrix}
0 & 1 & 0 & 0 \\
1 & 0 & 0 & 0 \\
0 & 0 & 0 & 0 \\
0 & 0 & 0 & 0
\end{pmatrix}
\gamma_j e_1
\]
\[
+ \gamma_i^T
\begin{pmatrix}
0 & 0 & 1 & 0 \\
0 & 0 & 0 & 0 \\
1 & 0 & 0 & 0 \\
0 & 0 & 0 & 0
\end{pmatrix}
\gamma_j e_2 + \gamma_i^T
\begin{pmatrix}
0 & 0 & 0 & 1 \\
0 & 0 & 0 & 0 \\
0 & 0 & 0 & 0 \\
1 & 0 & 0 & 0
\end{pmatrix}
\gamma_j e_3
\]
\[
= \gamma_i^T C
\begin{pmatrix}
\gamma_j &  &  & \\
& \gamma_j &  & \\
&  & \gamma_j & \\
&  & & \gamma_j
\end{pmatrix}
\begin{pmatrix}
e_0 \\
e_1 \\
e_2 \\
e_3
\end{pmatrix},
\]
so this yields the equation (\ref{eq}).
\end{proof}

The proof of the following lemma is straightforward.
\begin{lemma}\label{lem2}
Let $\mathcal{R}$ be a linear transformation on $\mathcal{H}_{d}$. The following equations
hold:
\begin{align*}
\mathcal{R}(e_i)e_j &= \gamma_i^T E_j^T
\begin{pmatrix}
e_0 \\
e_1 \\
e_2 \\
e_3
\end{pmatrix}, \\
e_i \mathcal{R}(e_j) &= \gamma_j^T E_i^T
\begin{pmatrix}
e_0 \\
e_1 \\
e_2 \\
e_3
\end{pmatrix}, \\
\end{align*}
where
\begin{align*}
E_0 &=
\begin{pmatrix}
1 & 0 & 0 & 0 \\
0 & 1 & 0 & 0 \\
0 & 0 & 1 & 0 \\
0 & 0 & 0 & 1
\end{pmatrix}, &
E_1 &=
\begin{pmatrix}
0 & 0 & 0 & 0 \\
1 & 0 & 0 & 0 \\
0 & 0 & 0 & 0 \\
0 & 0 & 0 & 0
\end{pmatrix}, \\
E_2 &=
\begin{pmatrix}
0 & 0 & 0 & 0 \\
0 & 0 & 0 & 0 \\
1 & 0 & 0 & 0 \\
0 & 0 & 0 & 0
\end{pmatrix}, &
E_3 &=
\begin{pmatrix}
0 & 0 & 0 & 0 \\
0 & 0 & 0 & 0 \\
0 & 0 & 0 & 0 \\
1 & 0 & 0 & 0
\end{pmatrix}.
\end{align*}
\end{lemma}
We now state our main theorem, which is a consequence of Lemmas \ref{lem1} and \ref{lem2}:
\begin{theorem}\label{th}
$\mathcal{R}$ is a Rota-Baxter operator of weight $\lambda$ on $\mathcal{H}_{d}$ if and only if the column vectors $\gamma_i$ of $R$ and $R$ satisfy the following equations:
$$
\gamma_i(4)^T C^T R = RE_iR + R(E_0, E_1, E_2, E_3)\gamma_i(4) + \lambda R E_i ,
$$
where $C$ and $E_i$ ($i = 0, 1, 2, 3$) are defined in Lemmas \ref{lem1} and \ref{lem2}.
\end{theorem}
\begin{proof}
Using Lemma \ref{lem2}, we obtain
\[
\mathcal{R}(\mathcal{R}(e_i)e_j) = \gamma_i^T E_j^T R^T
\begin{pmatrix}
e_0 \\
e_1 \\
e_2 \\
e_3
\end{pmatrix},
\]
\[
\mathcal{R}(e_i \mathcal{R}(e_j)) = \gamma_j^T E_i^T R^T
\begin{pmatrix}
e_0 \\
e_1 \\
e_2 \\
e_3
\end{pmatrix}.
\]
$\mathcal{R}$ is a Rota-Baxter operator with weight $\lambda$ if and only if for any $i, j = 0, 1, 2, 3$, the following equation holds:
\[
\mathcal{R}(e_i)\mathcal{R}(e_j) = \mathcal{R}\big(\mathcal{R}(e_i)e_j\big) + \mathcal{R}\big(e_i\mathcal{R}(e_j)\big) + \lambda \mathcal{R}(e_i e_j).
\]
For $j = 0$, we obtain from Lemma \ref{lem1} that
\[
\gamma_i^T C \gamma_0 (4) = \gamma_i^T R^T + \gamma_0^T E_i^T R^T + \lambda \gamma_i^T.
\]
This amounts to the first equality of Theorem.\\
When $j = 1$, Lemma \ref{lem1} gives
\[
\gamma_i^T C \gamma_1(4)
\begin{pmatrix}
e_0 \\
e_1 \\
e_2 \\
e_3
\end{pmatrix}
= \gamma_i^T E_1^T R^T
\begin{pmatrix}
e_0 \\
e_1 \\
e_2 \\
e_3
\end{pmatrix}
+ \gamma_1^T E_i^T R^T \begin{pmatrix}
e_0 \\
e_1 \\
e_2 \\
e_3
\end{pmatrix}
+ \lambda \mathcal{R}(e_ie_1).
\]
For $i=0, 1, 2, 3$, respectively, we have
\begin{align*}
\gamma_0^T C \gamma_1(4)&= \gamma_0^T E_1^T R^T + \gamma_1^T R^T + \lambda \gamma_1^T, \\
\end{align*}
\begin{align*}
\gamma_1^T C \gamma_1(4)&= 2\gamma_1^T E_1^T R^T, \\
\end{align*}
\begin{align*}
\gamma_2^T C \gamma_1(4)&= (\gamma_2^T E_1^T + \gamma_1^T E_2^T) R^T, \\
\end{align*}
\begin{align*}
\gamma_3^T C \gamma_1(4)&= (\gamma_3^T E_1^T + \gamma_1^T E_3^T) R^T, \\
\end{align*}
which is precisely the second equality in Theorem.\\
When $j = 2$, Lemma \ref{lem1} gives
\[
\gamma_i^T C \gamma_2(4)
\begin{pmatrix}
e_0 \\
e_1 \\
e_2 \\
e_3
\end{pmatrix}
= \gamma_i^T E_2^T R^T
\begin{pmatrix}
e_0 \\
e_1 \\
e_2 \\
e_3
\end{pmatrix}
+ \gamma_2^T E_i^T R^T \begin{pmatrix}
e_0 \\
e_1 \\
e_2 \\
e_3
\end{pmatrix}
+ \lambda \mathcal{R}(e_ie_2).
\]
For $i=0, 1, 2, 3$, respectively, we have
\begin{align*}
\gamma_0^T C \gamma_2(4)&= \gamma_0^T E_2^T R^T + \gamma_2^T R^T + \lambda \gamma_2^T, \\
\end{align*}
\begin{align*}
\gamma_1^T C \gamma_2(4)&= (\gamma_1^T E_2^T + \gamma_2^T E_1^T) R^T, \\
\end{align*}
\begin{align*}
\gamma_2^T C \gamma_2(4)&= 2\gamma_2^T E_2^T R^T, \\
\end{align*}
\begin{align*}
\gamma_3^T C \gamma_2(4)&= (\gamma_3^T E_2^T + \gamma_2^T E_3^T) R^T, \\
\end{align*}
which coincides with the third equality in Theorem.\\
When $j = 3$, from Lemma\ref{lem1}, one has
\[
\gamma_i^T C \gamma_3(4)
\begin{pmatrix}
e_0 \\
e_1 \\
e_2 \\
e_3
\end{pmatrix}
= \gamma_i^T E_3^T R^T
\begin{pmatrix}
e_0 \\
e_1 \\
e_2 \\
e_3
\end{pmatrix}
+ \gamma_3^T E_i^T R^T \begin{pmatrix}
e_0 \\
e_1 \\
e_2 \\
e_3
\end{pmatrix}
+ \lambda \mathcal{R}(e_ie_3).
\]
For $i=0, 1, 2, 3$, respectively, we have
\begin{align*}
\gamma_0^T C \gamma_3(4)&= \gamma_0^T E_3^T R^T + \gamma_3^T R^T + \lambda \gamma_3^T, \\
\end{align*}
\begin{align*}
\gamma_1^T C \gamma_3(4)&= (\gamma_1^T E_3^T + \gamma_3^T E_1^T) R^T, \\
\end{align*}
\begin{align*}
\gamma_2^T C \gamma_3(4)&= (\gamma_2^T E_3^T + \gamma_3^T E_2^T) R^T, \\
\end{align*}
\begin{align*}
\gamma_3^T C \gamma_3(4)&= 2\gamma_3^T E_3^T R^T, \\
\end{align*}
which is exactly the fourth equality in Theorem.
\end{proof}

\section{A Complete Characterization of Rota-Baxter Operators on $\mathcal{H}_d$}

This section is devoted to classifying all Rota–Baxter operators on $\mathcal{H}_d$ with non-zero weight $\lambda$. We begin by determining all $R$ that satisfy the matrix equations in Theorem~\ref{th}. By Theorem~\ref{th}, this matrix system reduces to the following system of equations:

\begin{eqnarray}
    a_{11}^2 +\lambda a_{11} &=& 0, \label{e1} \\
    2a_{22}a_{21} + 2a_{23}a_{31} + 2a_{24}a_{41}+\lambda a_{21} &=& 0, \label{e2} \\
    2a_{32}a_{21} + 2a_{33}a_{31} + 2a_{34}a_{41}+\lambda a_{31} &=& 0, \label{e3} \\
    2a_{42}a_{21} + 2a_{43}a_{31} + 2a_{44}a_{41}+\lambda a_{41} &=& 0, \label{e4}\\
    a_{22}^2+a_{23}a_{32}+a_{24}a_{42}+\lambda a_{22}&=&0, \label{e6}\\
    a_{32}a_{22}+a_{33}a_{32}+a_{34}a_{42}+\lambda a_{32}&=&0, \label{e7}\\
    a_{42}a_{22}+a_{43}a_{32}+a_{44}a_{42}+\lambda a_{42}&=&0, \label{e8}\\
    a_{22}a_{23}+a_{23}a_{33}+a_{24}a_{43}+\lambda a_{23}&=&0, \label{e10}\\
    a_{32}a_{23}+a_{33}^2+a_{34}a_{43}+\lambda a_{33}&=&0, \label{e11}\\
    a_{42}a_{23}+a_{43}a_{33}+a_{44}a_{43}+\lambda a_{43}&=&0, \label{e12}\\
    a_{22}a_{24}+a_{23}a_{34}+a_{24}a_{44}+\lambda a_{24}&=&0, \label{e14}\\
    a_{32}a_{24}+a_{33}a_{34}+a_{34}a_{44}+\lambda a_{34}&=&0, \label{e15}\\
    a_{42}a_{24}+a_{43}a_{34}+a_{44}^2+\lambda a_{44}&=&0, \label{e16}\\
    a_{12}&=&0, \label{e21}\\
    a_{13}&=&0, \label{e41}\\
    a_{14} &=&0, \label{e56}
\end{eqnarray}
From the equations \(a_{12}=a_{13}=a_{14}=0\), write \(R\) in block form
\[
R=\begin{pmatrix}
\alpha & \mathbf{0}^T\\
\mathbf b & A
\end{pmatrix},
\qquad
\alpha=a_{11},\quad
\mathbf b=\begin{pmatrix}a_{21}\\a_{31}\\a_{41}\end{pmatrix},\quad
A=\begin{pmatrix}
a_{22}&a_{23}&a_{24}\\
a_{32}&a_{33}&a_{34}\\
a_{42}&a_{43}&a_{44}
\end{pmatrix}.
\]
The given system is equivalent to the following block equations:
\begin{eqnarray}
\alpha^2+\lambda\alpha=0, \label{e26}\\
2A\mathbf b+\lambda \mathbf b=0, \label{e27}\\
A^2+\lambda A=0. \label{e28}
\end{eqnarray}
Indeed, equation (\ref{e27}) comes from equations (\ref{e2})-(\ref{e4}), while (\ref{e28}) comes from (\ref{e6})–(\ref{e16}).

\medskip

This leads to the following two cases:

\medskip

\noindent\textbf{Case 1}: Consider \(\lambda=0\). Then (\ref{e26}) gives \(\alpha=0\). Furthermore, equation (\ref{e27}) gives
$
A\mathbf b=0
$
and (\ref{e28}) reduces to
$
A^2=0.
$

\bigskip

This leads to one of the main results of this section.
\begin{theorem}
$\mathcal{R}$ is a Rota-Baxter operator of weight $0$ on $\mathcal{H}_d$ if and only if the matrix corresponding to $\mathcal{R}$ has the following form:
\[
R=\begin{pmatrix}
0 & \mathbf{0}^T\\
\mathbf b & A
\end{pmatrix},
\qquad
A^2=0,\quad A\mathbf b=0,
\]
where \(A\) is a \(3\times 3\) matrix and \(\mathbf b=(b_1,b_2,b_3)^T\).
\end{theorem}

\bigskip

\noindent\textbf{Case 2:} Consider $\lambda \neq 0$. From equation (\ref{e26}), we get
$
\alpha \in \{0,\,-\lambda\}.
$
Furthermore, equation (\ref{e28}) implies that the eigenvalues of $A$ lie in $\{0,-\lambda\}$. Consequently, the matrix $2A+\lambda I$ is invertible, since its eigenvalues are $\lambda$ and $-\lambda$, both nonzero.
Thus, the equation
\[
2A\mathbf b+\lambda \mathbf b=(2A+\lambda I)\mathbf b=0
\]
then forces
$
\mathbf b=0.
$

\medskip

The preceding reduction brings us to the second main theorem of this section, which we now state.
\begin{theorem}
$\mathcal{R}$ is a Rota-Baxter operator of weight $\lambda\neq 0$ on $\mathcal{H}_d$ if and
only if the matrix corresponding to $\mathcal{R}$ has the following form:
\[
R=
\begin{pmatrix}
\alpha & \mathbf{0}^T\\
\mathbf 0 & A
\end{pmatrix},
\qquad
\alpha \in \{0,\,-\lambda\},
\qquad
A^2+\lambda A=0.
\]
\end{theorem}

\vspace{1cm}

\textbf{Word Count:} Approximately $2100$ words (excluding references and equations).






\end{document}